\documentclass{amsart}
\usepackage[utf8]{inputenc}

\def\NZQ{\mathbb}               
\def\NN{{\NZQ N}}
\def\QQ{{\NZQ Q}}
\def\ZZ{{\NZQ Z}}

\def\F2{{\NZQ F}_2}

\def\opn#1#2{\def#1{\operatorname{#2}}} 
\opn\chara{char} \opn\length{\ell} \opn\pd{pd} \opn\rk{rk}
\opn\projdim{proj\,dim} \opn\injdim{inj\,dim} \opn\rank{rank}
\opn\depth{depth} \opn\codepth{codepth} \opn\grade{grade}
\opn\height{ht} \opn\embdim{emb\,dim} \opn\codim{codim}

\opn\Tr{Tr} \opn\bigrank{big\,rank}
\opn\superheight{superheight}\opn\lcm{lcm}
\opn\trdeg{tr\,deg}%
\opn\reg{reg} \opn\lreg{lreg} \opn\skel{skel}
\opn\Gr{Gr}
\opn\ann{ann}
\opn\sign{sign}
\opn\del{del}
\opn\lex{lex}

\opn\div{div} \opn\Div{Div} \opn\cl{cl} \opn\Cl{Cl}
\opn\Spec{Spec} \opn\Supp{Supp} \opn\supp{supp} \opn\Sing{Sing}
\opn\Ass{Ass}\opn\fdepth{fdepth}
\opn\Ann{Ann} \opn\Rad{Rad} \opn\Soc{Soc}
\opn\Sym{Sym} \opn\Ker{Ker} \opn\Coker{Coker} \opn\Im{Im}
\opn\Hom{Hom} \opn\Tor{Tor} \opn\Ext{Ext} \opn\End{End}
\opn\Aut{Aut} \opn\id{id} \opn\ini{in} \opn\tr{tr}

\opn\nat{nat}\opn\it{it}
\opn\pff{proof}
\opn\Pf{proof} \opn\GL{GL} \opn\SL{SL} \opn\mod{mod} \opn\ord{ord}
\opn\aff{aff} \opn\con{conv} \opn\relint{relint} \opn\st{st}
\opn\lk{lk} \opn\cn{cn} \opn\core{core} \opn\vol{vol}
\opn\link{link} \opn\star{star} \opn\skel{skel} \opn\indeg{indeg}
\opn\Ass{Ass} \opn\Min{Min} \opn\sdepth{sdepth} \opn\depth{depth}
\opn\gr{gr}

\def\pot#1#2{#1[\kern-0.28ex[#2]\kern-0.28ex]}

\opn\dirlim{\underrightarrow{\lim}}
\opn\inivlim{\underleftarrow{\lim}}
\let\to=\rightarrow

\def\Implies{\ifmmode\Longrightarrow \else
     \unskip${}\Longrightarrow{}$\ignorespaces\fi}
\def\implies{\ifmmode\Rightarrow \else
     \unskip${}\Rightarrow{}$\ignorespaces\fi}
\def\iff{\ifmmode\Longleftrightarrow \else
     \unskip${}\Longleftrightarrow{}$\ignorespaces\fi}

\let\:=\colon

\let\ol=\overline
\theoremstyle{plain}
\newtheorem{Theorem}{Theorem}[section]
 \newtheorem{Lemma}[Theorem]{Lemma}
 \newtheorem{Corollary}[Theorem]{Corollary}
 \newtheorem{Proposition}[Theorem]{Proposition}

 \theoremstyle{definition}
 \newtheorem{Definition}[Theorem]{Definition}
  
 \newtheorem{Remark}[Theorem]{Remark}
 
 \newtheorem{Example}[Theorem]{Example}

\let\epsilon\varepsilon
\let\kappa=\varkappa
\opn\dis{dis}
\def\pnt{{\raise0.5mm\hbox{\large\bf.}}}

\opn\Lex{Lex}

\usepackage{subcaption,tikz}
\usetikzlibrary{patterns}

\usepackage{float}
\usepackage{tikz-cd}
\usepackage{mathtools}
\usepackage{amsfonts}
\usepackage{amssymb}
\newcommand{\PP}{\mathcal{P}}
\newcommand{\FF}{\mathbb{F}}
\newcommand{\MF}{\mathcal{F}}
\newcommand{\KK}{\mathbb{K}}

\renewcommand{\phi}{\varphi}

\renewcommand{\H}{\mathrm{HF}}
\newcommand{\cha}{\mathrm{char}}

\renewcommand{\top}{\mathrm{top}}
\newcommand{\Triv}{\mathrm{Triv}}
\newcommand{\triv}{\mathrm{triv}}
\newcommand{\HP}{\mathrm{HP}}
\newcommand{\HS}{\mathrm{HS}}

\renewcommand{\SS}{\mathcal{S}}

\newcommand{\MV}{\mathcal{V}}

\newcommand{\MM}{\mathcal{M}}

\renewcommand{\QQ}{\mathcal{Q}}

\usepackage{enumitem}

\title{Hilbert series and degrees of regularity of Oil $\&$ Vinegar and Mixed quadratic systems}
\author{Antonio Corbo Esposito, Rosa Fera, Francesco Romeo}
\address{University of Cassino and Southern Lazio \\ 
DIEI, Department of Electrical and Information Engineering}
\email{corbo@unicas.it, rosa.fera@unicas.it, francesco.romeo@unicas.it}
\date{}

\begin{document}

\maketitle

\begin{abstract}
    In this paper, we analyze the algebraic invariants for two classes of multivariate quadratic systems: systems made by OV quadratic polynomials and systems made by both OV polynomials and fully quadratic ones. For such systems, we explicitly compute the Hilbert series, and we give bounds on the degree of regularity, solving degree and first fall degree that are useful in cryptographic applications.
\end{abstract}

\section{Introduction}
The problem of solving multivariate polynomial systems is one of the hard problems that have been used to build quantum-resistant cryptosystems. 
It is interesting to explore how properties, invariants and tools in Commutative Algebra can provide insights into the complexity of solving these problems. In the modern literature, to measure such complexity, various authors have introduced quantities associated with different algebraic properties: the \emph{first fall degree} that is related to the computation of the syzygies of the system, the \emph{degree of regularity} that is computed through the Hilbert series, and the \emph{solving degree} that arises from Gr\"obner bases. In particular, Syzygies represent relationships between polynomials to spot redundancies or dependencies among them, Hilbert series of a polynomial ring captures informations about the dimensions of the vector spaces of homogeneous polynomials in different degrees, while Gröbner bases are sets of polynomials from which one can derive all the polynomials in an ideal, and they have a specific ``leading term" relationship that simplifies many algebraic operations and computations.

Several of the most recent multivariate polynomial schemes, proposed for standardization to NIST competition, derive their security from quadratic polynomial systems over finite fields.
This problem is known as Multivariate Quadratic (MQ) problem.  
In 1997, Patarin proposed the \textit{Oil and Vinegar scheme} (OV scheme) \cite{Pa3}. The main feature of these multivariate quadratic systems relies on separating the variables into two distinct subsets: the subset of variables called ``oil'' and the one of the variables called ``vinegar''. Each polynomial (of degree 2 with coefficients in a finite field) is chosen such that, after an evaluation in the vinegar variables, it is linear only in the oil variables. Such a scheme is based on an under-determined MQ system and the goal is to find one of the many preimages of the original message.
It must be highlighted that in the case in which the number of the oil variables is equal to the number of the vinegar variables, the scheme was already broken by Kipnis and Shamir (see \cite{KS1}), therefore it was modified to the so-called ``Unbalanced Oil and Vinegar scheme'' (UOV), where the number of vinegar variables is greater than the number of oil ones (see \cite{KPG}). UOV is still used in many of the recent multivariate schemes (see \cite{DS,Be1}).


In this paper, we analyze the algebraic properties of two kinds of systems: one comprising only OV polynomials, and the other including both OV polynomials and fully quadratic ones.  For this kind of schemes we compute the Hilbert series and other algebraic invariants.  The Hilbert series of OV-systems is also studied in \cite{IS}, where the authors give a lower bound of its coefficients. We improve the latter result by giving an exact formula as sum of two other Hilbert series. The crucial observation is that the ideal generated by a homogeneous OV system is contained in the ideal $\MV$ generated by the vinegar variables. The study of the Hilbert series of the OV-system $\MF$ in the polynomial ring $R$ reduces to the study of the Hilbert series of $\MF$ in the ideal $\MV$ and we introduce a degree of regularity for OV-systems. We also focus on fields with  characteristic 0: we compute some algebraic invariants of $R/\MF$, we give an upper bound for the degree of regularity in the over-determined case and we also discuss some relations arising from non-regular sequences, and we give definitions of OV semi-regular sequence; moreover, we show that any quadratic system can be seen as an OV-system up to linear transformation.
We use similar techniques to compute the Hilbert series of mixed systems, observing that the fully quadratic equations can be split into a part in $\MV$, involved in the computation with $\MF$, and in an oil part that gives contribution to the Hilbert series separately in the polynomial ring in the sole oil variables. 
We also analyze OV non-homogeneous quadratic systems and we discuss their first fall degree, degree of regularity and solving degree.


The paper is organized as follows. In Section \ref{sec:Pre}, we recall several definitions that are useful for the sequel; in particular, the algebraic invariants, the cryptographic background and direct computations of the Hilbert series of semi-regular quadratic systems (see \ref{def:semireg}).  In Section \ref{sec:HSFF}, we give a definition of semi–regular sequence on a finite field $\FF_q$ and we compute the Hilbert series for the related systems. In Section \ref{sec:HSOV}, we dive into OV homogeneous quadratic systems: we compute their Hilbert series, by using additive property, as the sum of the Hilbert series of the polynomial ring in the oil variables and the Hilbert series of the vinegar ideal modulo the ideal generated by the system (Theorem \ref{thm:HFOV});
in Section \ref{sec:char0}, by restricting our attention on characteristic-0 fields, we compute height and Krull dimension of the ideal $\MF$ generated by the system, and since $\MF \subset \MV$, these two invariants are both related to $v$ (where $v$ is the number of the vinegar variables) (Lemma \ref{lem:algInv}); in particular if $\MF=(f_1,\ldots, f_m)$, then $f_1,\ldots, f_m$ can be regular if and only if $m \leq v$. The latter implies that for any $m\geq v+1$, there are trivial relations arising from any subset of cardinality $v+1$ of the $m$ equations; such relations appear at degree $v+2$ as pointed out in Proposition \ref{prop:v+2}. Moreover if $m\geq n$, we prove that the degree of regularity is bounded by $v+1$ (Proposition \ref{prop:dreg}). Thanks to the previous results, we give definitions of semi-regularity for OV systems. Furthermore, in Section \ref{sec:HSQ}, by using Noether Normalization and Relative Finiteness Theorem, we show that any quadratic system $\MF$ can be seen as an OV-system in $v=n-d$ vinegar variables where $d$ is the Krull dimension of $R/\MF$.

In Section \ref{sec:Mix}, we introduce mixed systems $\PP$, that consist in an OV system $\MF$ and in a fully quadratic system $Q$. Given the vinegar ideal $\MV$, the ideal $\MV + Q$ can be seen as $\MV+ Q_o$, where $Q_o$ is generated by the oil monomials of the polynomials in $Q$, and it is contained in the polynomial ring $\KK_o=\KK[x_{v+1},\ldots, x_n]$, of the oil variables. Hence, the Hilbert series of $R/\PP$ can be splitted as a sum of the Hilbert series of $(\MV+Q_o)/\PP$ and the Hilbert series of $\KK_o/Q_o$ (Theorem \ref{thm:HFP}). Even for those systems we introduce a degree of regularity that, in analogy with the OV system, depends on the index of regularity of one of the summands of the Hilbert Series. 
In Section \ref{sec:HSNH}, we discuss the solving degree and the first fall degree of OV and cryptographic semiregular non-homogeneous quadratic systems.

\section{Preliminaries}\label{sec:Pre}

\subsection{Algebraic Setting}\label{sec:HS&Alg}

Let $R$ be a ring.
The \emph{Krull dimension} of $R$, $\dim R$, is the supremum of the lengths of all chains of prime ideals of $R$.
\noindent In other words if 
\[
\mathcal{C}: \  \mathfrak{p}_{0} \subset \mathfrak{p}_{1} \subset \ldots \subset \mathfrak{p}_{n} 
\]
is a chain of prime ideals of $R$, we say that $C$ has length $n$. If the supremum of the lengths is not finite, we say that the Krull dimension is $+\infty$. 
As an example if $\mathbb{K}$ is a field then $\dim \mathbb{K}=0$ since $(0)$ is the unique prime ideal in any field.
In the whole work we always denote by $\dim$ the Krull dimension and by $\dim_\KK$ the dimension of a $\KK$-vector space. 

Let $\mathfrak{p}$ be a prime ideal of $R$. The height of $\mathfrak{p}$ is defined as the largest number $h$ such that there exists a chain of different prime ideals 
\[
\mathfrak{p}_{0} \subset \mathfrak{p}_{1} \subset \ldots \subset \mathfrak{p}_{h}=\mathfrak{p} 
\]
descending from $\mathfrak{p}$.
Let $I$ be an ideal of a commutative ring $R$, the  \emph{height} of $I$, $\height(I)$ is defined to be the minimum of the heights of the prime ideals containing the ideal $I$.

Let $M$ be an $R$-module. An element $r \in R\setminus \{0\}$ is a \textit{non-zero divisor on $M$} if $rm=0$ implies $m=0$ for $m \in M$.
A \emph{regular sequence} is a sequence of elements $r_{1},\ldots ,r_{d}\in R$ such that:
\begin{itemize}
\item[a)] $r_{1}$ is a non-zero divisor on $M$ and for any $i \in \{2,\ldots,d\}$ the element $r_{i}$ is a non-zero divisor on $M/(r_{1},\ldots ,r_{i-1})M$;
\item[b)] $(r_{1},\ldots ,r_{d})M\neq M$.
\end{itemize} 

Let $I,J \subseteq R$ be ideals. The \emph{colon ideal} is 
\[
(I:J)=\{f \in R : fJ \subset I\}.
\]

Let $R=\KK[x_1,\ldots, x_n]$ be a standard graded polynomial ring on the field $\KK$, and we denote by $\mathrm{char}{(\KK)}$ its characteristic.
Given $I\subseteq R$ homogeneous ideal,   the \emph{Hilbert function} $\H_{R/I} : \mathbb{N} \rightarrow \mathbb{N}$ is defined by 
\[
\H_{R/I} (d) := \dim_\KK (R/I)_d
\]
where $(R/I)_d$ is the $d$-degree component of the gradation of $R/I$.  Hilbert function is known to be  additive with respect to exact sequences.

\begin{Lemma}\label{lem:addHS}
    Let 
    \[
    0\longrightarrow A \longrightarrow B \longrightarrow C \longrightarrow 0
    \]
    be an exact sequence of $R$-modules. Then, for any $d \in \NN$
    \[
    \H_{B}(d)=\H_{A}(d)+\H_{C}(d).
    \]
\end{Lemma}
From Lemma \ref{lem:addHS}, it follows that adding an equation to a given system $\MF$ results in a decreasing of the Hilbert function, as in the following
\begin{Remark}\label{rem:FF'}
    Let $\MF=\{f_{1},\ldots, f_{m}\}$ be homogeneous system of polynomials, let $\MF'=\{f_1,\ldots, f_{m+1}\}$ with $\deg f_m = \delta$. The two systems are related by the following exact sequence
    \[
    0\to (R/(\MF:f_{m+1}))_{d-\delta}\to (R/\MF)_d \to (R/\MF')_{d} \to 0
    \]
    It follows that 
    \[
    \H_{R/\MF'}(d)= \H_{R/\MF}(d) - \H_{R/(\MF:f_{m+1})}(d-\delta),
    \]
    in particular
    \[
    \H_{R/\MF'}(d)\leq \H_{R/\MF}(d)
    \]
\end{Remark}

Given a formal power series $\sum_{d} a_dt^d$, we define $\Bigg[\sum_{d} a_dt^d \Bigg]=\sum_d b_d t^d$ where $b_d=a_d$ if $a_i>0$ for any $0\leq i \leq d$ and $b_d=0$, otherwise.
The \emph{Hilbert-Poincar\'e series} of $R/I$ is
\[
\HS_{R/I} (t) := \sum_{d \in \NN} \H_{R/I}(d) t^d. 
\]
By the Hilbert-Serre theorem, the Hilbert-Poincar\'e series of $R/I$ is generated by the $[\cdot]$ applied to the power series expansion of a rational function. We call such a function \emph{generating function}. In particular, by reducing this rational function we get
\[
\HS_{R/I}(t) = \Bigg[ \frac{h(t)}{(1-t)^{\dim R/I}}\Bigg]. 
\]
that is called  \emph{reduced Hilbert series}, where the $[ \cdot ]$ on right-hand-side is intended on its power series expansion. The notation with $[ \cdot ]$ has to be adopted because the formal power series arising from the right-hand-side may have negative coefficients, that contradicts the definition of Hilbert function. However, we highlight that  the expansion of the series with negative coefficients is useful to understand the behaviour of the systems in non-homogeneous case (see Lemma \ref{lem:SD} and Proposition \ref{prop:sdOV}). From now on, in the examples we fill free to use the expression $\HS$ to denote both the generating function of the series. 

It is well-known that for large $d$ the Hilbert function $\H_{R/I}(d)$ is a polynomial, called \emph{Hilbert polynomial}, denoted by $\HP_{R/I}(t)$. The least $d$ for which $\H_{R/I}(d)=\HP_{R/I}(d)$, is called \emph{index of regularity}, denoted by $i_{\reg{}}(I)$. 
Observe that, $\H_{R/I}(d)= \dim_\KK R_d - \dim_\KK I_d$. According to this expression, in the (over)determined case, the index of regularity can be seen as the smallest $d$ for which $I_d=R_d$. For further information on Hilbert series see \cite{KR}.

\subsection{MQ-problem and OV-polynomials}

In this section, we recap fundamental notions and definitions that are useful in the sequel. 
Our recurring object will be a quadratic polynomial on $\KK[x_1,\ldots, x_n]$, say 
\[
p(x_{1},\ldots, x_{n})=\sum\limits_{1\leq  i \leq j \leq n} \alpha_{ij} x_{i}x_{j} +\sum\limits_{i=1}^n \beta_{i} x_i + \gamma 
\]
Without any particular assumption on the $\alpha_{ij}$, we will call such a polynomial a \emph{fully quadratic} polynomial.

Moreover, let $v, n \in \NN$ be positive integers with $v< n$. A quadratic polynomial of the form
\begin{equation}\label{eq:OV}
    f(x_{1},\ldots, x_{n})=\sum\limits_{i=1}^n \ \ \sum\limits_{j=1}^v \alpha_{i,j} x_i x_j +\sum\limits_{i=1}^n \beta_{i} x_i + \gamma,
\end{equation}
namely with no monomial $x_{i}x_{j}$ for $i,j \in \{v+1,\ldots,n\}$, is called an \emph{Oil and Vinegar (OV)} polynomial. The variables $x_1, \ldots, x_v$ are called \emph{vinegar} variables, while the variables $x_{v+1},\ldots, x_{n}$ are called \emph{oil} variables. We set $\mathcal{V}=(x_1,\ldots,x_v)$, namely the ideal generated by the vinegar variables. For an OV-system with $v$ vinegar variables, we set $\KK_o=\KK[x_{v+1}, \ldots, x_{n}]$, to be the polynomial ring in the oil variables

Given a non-homogeneous quadratic polynomial $f$, we call $f^{\top}$ the homogeneous polynomial corresponding to the highest-degree monomials of $f$. Given a system of polynomials $\MF=\{f_1,\ldots, f_{m}\}$ we call $\MF^\top$ the system $\{f^\top_1,\ldots,f^\top_m\}$.

\subsection{Hilbert series of Multivariate polynomials and degree of regularity}
We define the \emph{degree of regularity} $d_{\reg{}}(\MF)$ of a system $\MF$ of polynomial equations as the index of regularity of $R/\MF$. 
From the characterization of regular sequences (see \cite[Proposition 1.7.4]{Ba}), we can obtain the following on quadratic polynomials.
\begin{Proposition}\label{prop:regSeq}
    Let $f_1,\ldots, f_m$ be a homogeneous regular sequence of quadratic polynomials with $\mathrm{char}(\KK)=0$. The following are equivalent:
    \begin{itemize}
        \item the ideal $(f_1,\ldots, f_m )$ has Krull dimension $n-m$;
        \item the Hilbert series of $R/( f_1,\ldots, f_m )$ is
        \[
         \frac{(1-t^2)^m}{(1-t)^n}
        \]
    \end{itemize}
\end{Proposition}

\begin{Definition}
    Let $R = \KK[x_1,\ldots,x_n]$ and assume that $\KK$ is an infinite field with $\mathrm{char}(\KK)=0$. If $A = R/I$, where $I$ is a homogeneous ideal, and $f \in R_d$, then $f$ is \textbf{semi-regular} on $A$ if, for every $e \geq d$, the vector space map $A_{e-d} \to A_e$ given by multiplication by $f$ is of maximal rank (that is, either injective or surjective). A sequence of homogeneous polynomials $f_1,\ldots,f_m$ is a \textbf{semi-regular sequence} if each $f_i$ is semi-regular on $A/(f_1,\ldots,f_{i-1})$, for $1 \leq i \leq m$.
\end{Definition}

For semi-regular sequence there is a result similar to Proposition \ref{prop:regSeq}.
\begin{Proposition}[\cite{Pa}, Proposition 1]
    Let $\KK$ be an arbitrary field and set $R = \KK[x_1,\ldots,x_n]$. Let $f_1, \ldots, f_m \in R $ be homogeneous quadratic polynomials. Then $f_1, \ldots, f_m$ is a semi-regular sequence on $R$ if and only if
    \[
    \HS_{R/(f_1,\ldots,f_l)}(t) = \Bigg[ \frac{(1-t^2)^l}{(1-t)^n} \Bigg]
    \]
    for every $1 \leq l \leq m$.
\end{Proposition}

In \cite{BFS}, the authors adopt a concept of semi-regularity that is useful for cryptographic applications. 

\begin{Definition}\label{def:semireg}
A sequence of quadratic homogeneous polynomials $f_1,\ldots, f_m$  is called \emph{cryptographic semi-regular} if 
\[
\HS_{R/(f_1,\ldots, f_m)}(t)= \Bigg[ \frac{(1-t^2)^m}{(1-t)^n} \Bigg]
\]
\end{Definition}

In \cite{Ba}, it is showed  that Proposition \ref{prop:regSeq} and Definition \ref{def:semireg} can be restated on $\FF_2$ as follows
\begin{Proposition}\label{prop:regSeq2}
    Let $f_1,\ldots, f_m$ be a homogeneous regular sequence of quadratic polynomials on $\FF_2$. The following are equivalent:
    \begin{itemize}
        \item the ideal $(f_1,\ldots, f_m )$ has Krull dimension $n-m$;
        \item the Hilbert series of $R/( f_1,\ldots, f_m )$ is
        \[
          \Bigg[ \frac{(1+t)^n}{(1+t^2)^m} \Bigg]
        \]
    \end{itemize}
    
\end{Proposition}

\begin{Definition}
A sequence of quadratic homogeneous polynomials $f_1,\ldots, f_m$ on $\FF_2$  is called \emph{cryptographic semiregular} if 
\[
\HS_{R/(f_1,\ldots, f_m)}(t)=  \Bigg[ \frac{(1+t)^n}{(1+t^2)^m} \Bigg]
\]
\end{Definition}

 \begin{Example}
 \begin{itemize}
     \item      We consider $\FF_{2}[x_1,\ldots, x_{8}]$ and we take a polynomial $f(x)$ of degree $2$. Then,
     \[
\HS(t)=\frac{(1+t)^8}{1+t^2}=1+8t+27t^2+48t^3+43t^4+8t^5-15t^6.
\]
We compute the number of monomials and the number of equations at degree $d \geq 2$ until we find a negative difference. \\
For $d=2$, there are  $\binom{8}{2}=28$ monomials of degree $2$ and one quadratic equation, hence $28-1=27$ is the degree-$2$ coefficient of the Hilbert series.\\
For $d=3$, we have $\binom{8}{3}=56$ monomials of degree $3$ and $8$ equations $x_i f$ for $i \in \{1,\ldots, 8\}$. They are all independent  and hence $56-8=48$ is the degree-$3$ coefficient of the Hilbert series.\\
For $d=4$ we have $\binom{8}{4}=70$ monomials, and $28$ equations of the form $x_ix_j f$ for $i, j \in \{1,\ldots, 8\}$ and $i\leq j$. In this case, one obtains the trivial equation $f^2=f$, hence the number of non-trivial equations is $27$ and $70-27=43$ is the degree-$4$ coefficient of the Hilbert series. \\
For $d=5$ we have $\binom{8}{5}=56$ monomials, and $56$ equations of the form $x_ix_jx_k f$ for $i, j,k \in \{1,\ldots, 8\}$ and $i\leq j\leq k$. Among them, we have $8$ trivial equations that are $x_i f^2=x_i f$ for $i \in \{1,\ldots, 8\}$, that is the number of non-trivial equations is $48$ and $56-48=8$ is the degree-$5$ coefficient of the Hilbert series. \\
For $d=6$ we have $\binom{8}{6}=28$ monomials, and $70$ equations. Among them, we have $28$ trivial equations that are $x_ix_j f^2=x_ix_j f$ for $i,j \in \{1,\ldots, 8\}$, whose 27 are independent ($f^3=f^2=f$ is dependent), that is the number of non-trivial equations is $43$ and $28-43=-15$ is the degree-$6$ coefficient of the Hilbert series. \\
In this case, we see that the degree of regularity is $6$. 
\item We now consider a ``cryptographic semiregular'' system of $13$ quadratic equations in $\FF_{2}[x_1,\ldots, x_{10}]$. Then,
\[
\HS(t)=\frac{(1+t)^{10}}{(1+t^2)^{13}}=1+10t+32t^2-10t^3-284t^4.
\]
For $d=2$, there are  $\binom{10}{2}=45$ monomials of degree $2$ and $13$ quadratic equations, hence $45-13=32$ is the degree-$2$ coefficient of the Hilbert series.\\
For $d=3$, we have $\binom{10}{3}=120$ monomials of degree $3$ and $130$ equations $x_i f_j$ for $i \in \{1,\ldots, 10\}$ and $j \in \{1,\ldots, 13\}$. They are all independent  and hence $120-130=-10$ is the degree-$3$ coefficient of the Hilbert series.\\
For $d=4$ we have $\binom{10}{4}=210$ monomials, and $585$ equations of the form $x_ix_j f_k$ for $i, j \in \{1,\ldots, 10\}$, $i\leq j$ and $k \in \{1,\ldots 13\}$. In this case, one obtains the $13$ trivial equations $f_k^2=f_k$ and the $\binom{13}{2}=78$ trivial equations $f_if_j=f_j f_i$, hence the number of non-trivial equations is $585-13-78=494$ and $210-494=-284$ is the degree-$4$ coefficient of the Hilbert series.

 \end{itemize}
 \end{Example}
 
\subsection{Macaulay Matrix and solving degree}
Let us consider a system $\MF=\{f_{1}, \ldots f_{m}\}$ of $m$ quadratic equations in $n$ variables.
The \emph{homogeneous Macaulay Matrix} $M_d$ of $\MF$ at degree $d$ is the matrix whose columns are all the degree-$d$ monomials indexed in descending order from left to right with respect to a given term order, and the  rows are obtained by multiplying any $f_i$ for a monomial $m_i$ with $\deg m_i= d- \deg f_{i}$.
Any entry corresponds to the coefficient of the monomial, on the column, in the corresponding equation, on the row. 
 Similarly, for a non-homogeneous system $\MF$, the \emph{Macaulay Matrix} at degree $d$, $M_{\leq d}$, is the matrix whose columns are labelled by the monomials up to degree $d$, in descending order from left to right, and the rows are obtained by multiplying any $f_i$ for a monomial $m_i$ with $\deg m_i\leq  d- \deg f_{i}$. Similarly to the homogeneous case, any entry corresponds to the coefficient of the monomial, on the column, in the corresponding equation, on the row. 
 The least degree $d$ such that, after Gaussian elimination, the rows of $M_{\leq d}$ form a Gr\"obner basis with respect a term order $<$ is called \emph{solving degree} and it is denoted by $\mathrm{solv.deg_{<}(\MF)}$.

\begin{Example}\label{exa:MM}
We consider a system of $3$ polynomials in the polynomial ring $\FF_2[x,y,z]$ with the ordering \texttt{grevlex}, 
    \begin{align*}        
    &f_1=x^2 + xy +y^2+ xz   +x + y + z\\
    &f_2=x^2  + y^2 + yz + z^2 + x+ 1\\
    &f_3=x^2 + y^2 + xz + x  + y + 1\\
    &f_4=x^2 + y^2+ z^2+  x + y  + z\\
    \end{align*}
    The matrix $M_{\leq 3}$ for these polynomials is the following:
\[
\begin{array}{c}
\begin{matrix}
&x^3&x^2y&xy^2&y^3&x^2z&xyz&y^2z&xz^2&yz^2&z^3&x^2&xy&y^2&xz&yz&z^2&x&y&z&1\\
f_1&0 &0 &0 &0 &0 &0 &0 &0 &0 &0 &1 &1 &1 &1 &0 &0 &1 &1 &1 &0 \\
f_2&0 &0 &0 &0 &0 &0 &0 &0 &0 &0 &1 &0 &1 &0 &1 &1 &1 &0 &0 &1 \\
f_3&0 &0 &0 &0 &0 &0 &0 &0 &0 &0 &1 &0 &1 &1 &0 &0 &1 &1 &0 &1 \\
f_4&0 &0 &0 &0 &0 &0 &0 &0 &0 &0 &1 &0 &1 &0 &0 &1 &1 &1 &1 &0 \\
zf_1&0 &0 &0 &0 &1 &1 &1 &1 &0 &0 &0 &0 &0 &1 &1 &1 &0 &0 &0 &0 \\
yf_1&0 &1 &1 &1 &0 &1 &0 &0 &0 &0 &0 &1 &1 &0 &1 &0 &0 &0 &0 &0 \\
xf_1&1 &1 &1 &0 &1 &0 &0 &0 &0 &0 &1 &1 &0 &1 &0 &0 &0 &0 &0 &0 \\
zf_2&0 &0 &0 &0 &1 &0 &1 &0 &1 &1 &0 &0 &0 &1 &0 &0 &0 &0 &1 &0 \\
yf_2&0 &1 &0 &1 &0 &0 &1 &0 &1 &0 &0 &1 &0 &0 &0 &0 &0 &1 &0 &0 \\
xf_2&1 &0 &1 &0 &0 &1 &0 &1 &0 &0 &1 &0 &0 &0 &0 &0 &1 &0 &0 &0 \\
zf_3&0 &0 &0 &0 &1 &0 &1 &1 &0 &0 &0 &0 &0 &1 &1 &0 &0 &0 &1 &0 \\
yf_3&0 &1 &0 &1 &0 &1 &0 &0 &0 &0 &0 &1 &1 &0 &0 &0 &0 &1 &0 &0 \\
xf_3&1 &0 &1 &0 &1 &0 &0 &0 &0 &0 &1 &1 &0 &0 &0 &0 &1 &0 &0 &0 \\
zf_4&0 &0 &0 &0 &1 &0 &1 &0 &0 &1 &0 &0 &0 &1 &1 &1 &0 &0 &0 &0 \\
yf_4&0 &1 &0 &1 &0 &0 &0 &0 &1 &0 &0 &1 &1 &0 &1 &0 &0 &0 &0 &0 \\
xf_4&1 &0 &1 &0 &0 &0 &0 &1 &0 &0 &1 &1 &0 &1 &0 &0 &0 &0 &0 &0 \\

\end{matrix}
\end{array}
\]
We apply Gaussian Elimination and we obtain   
    \[
\begin{array}{c}
\begin{matrix}
&x^3&x^2y&xy^2&y^3&x^2z&xyz&y^2z&xz^2&yz^2&z^3&x^2&xy&y^2&xz&yz&z^2&x&y&z&1\\
&1 &0 &0 &0 &0 &0 &0 &0 &0 &1 &0 &0 &0 &0 &0 &1 &0 &0 &0 &1 \\
&0 &1 &0 &0 &0 &0 &0 &0 &0 &0 &0 &0 &0 &0 &0 &1 &0 &0 &0 &1 \\
&0 &0 &1 &0 &0 &0 &0 &0 &0 &0 &0 &0 &0 &0 &0 &0 &0 &0 &1 &1 \\
&0 &0 &0 &1 &0 &0 &0 &0 &0 &0 &0 &0 &0 &0 &0 &0 &0 &0 &0 &0 \\
&0 &0 &0 &0 &1 &0 &0 &0 &0 &1 &0 &0 &0 &0 &0 &0 &0 &0 &1 &1 \\
&0 &0 &0 &0 &0 &1 &0 &0 &0 &0 &0 &0 &0 &0 &0 &1 &0 &0 &1 &0 \\
&0 &0 &0 &0 &0 &0 &1 &0 &0 &0 &0 &0 &0 &0 &0 &0 &0 &0 &1 &1 \\
&0 &0 &0 &0 &0 &0 &0 &1 &0 &1 &0 &0 &0 &0 &0 &1 &0 &0 &1 &0 \\
&0 &0 &0 &0 &0 &0 &0 &0 &1 &0 &0 &0 &0 &0 &0 &1 &0 &0 &0 &1 \\
&0 &0 &0 &0 &0 &0 &0 &0 &0 &0 &1 &0 &0 &0 &0 &1 &0 &0 &0 &0 \\
&0 &0 &0 &0 &0 &0 &0 &0 &0 &0 &0 &1 &0 &0 &0 &0 &0 &0 &1 &1 \\
&0 &0 &0 &0 &0 &0 &0 &0 &0 &0 &0 &0 &1 &0 &0 &0 &0 &0 &0 &0 \\
&0 &0 &0 &0 &0 &0 &0 &0 &0 &0 &0 &0 &0 &1 &0 &1 &0 &0 &1 &1 \\
&0 &0 &0 &0 &0 &0 &0 &0 &0 &0 &0 &0 &0 &0 &1 &0 &0 &0 &1 &1 \\
&0 &0 &0 &0 &0 &0 &0 &0 &0 &0 &0 &0 &0 &0 &0 &0 &1 &0 &1 &0 \\
&0 &0 &0 &0 &0 &0 &0 &0 &0 &0 &0 &0 &0 &0 &0 &0 &0 &1 &0 &0 \\
\end{matrix}
\end{array}
\]
Hence, we obtain the following set of polynomials
\begin{align*}
   \{&x^3 + z^3 + z^2 + 1, \ \  x^2y + z^2 + 1, \ \  xy^2 + z + 1, \ \  y^3,\\
    &x^2z + z^3 + z + 1,  \ \ xyz + z^2 + z, \ \  y^2z + z + 1, \ \  xz^2 + z^3 + z^2 + z,\\
    &yz^2 + z^2 + 1, \ \  x^2 + z^2,  \ \ xy + z + 1, \ \  y^2,\\
    &xz + z^2 + z + 1,  \ \ yz + z + 1, \ \  x + z, \ \  y\}.
\end{align*}
and this is not a Gr\"obner basis because
\[
S(yz + z + 1,y) \to z+1. 
\]
Hence, we continue the computation to the degree 4. We compute the Matrix $M_{\leq 4}$, we apply Gaussian elimination, to get 
\begin{center}
\resizebox{1\textwidth}{!}{$
\begin{matrix}
&x^4&x^3y&x^2y^2&xy^3&y^4&x^3z&x^2yz&xy^2z&y^3z&x^2z^2&xyz^2&y^2z^2&xz^3&yz^3&z^4&x^3&x^2y&xy^2&y^3&x^2z&xyz&y^2z&xz^2&yz^2&z^3&x^2&xy&y^2&xz&yz&z^2&x&y&z&1\\
&1 &0 &0 &0 &0 &0 &0 &0 &0 &0 &0 &0 &0 &0 &1 &0 &0 &0 &0 &0 &0 &0 &0 &0 &0 &0 &0 &0 &0 &0 &0 &0 &0 &0 &0 \\
&0 &1 &0 &0 &0 &0 &0 &0 &0 &0 &0 &0 &0 &0 &0 &0 &0 &0 &0 &0 &0 &0 &0 &0 &1 &0 &0 &0 &0 &0 &1 &0 &0 &0 &0 \\
&0 &0 &1 &0 &0 &0 &0 &0 &0 &0 &0 &0 &0 &0 &0 &0 &0 &0 &0 &0 &0 &0 &0 &0 &0 &0 &0 &0 &0 &0 &1 &0 &0 &0 &1 \\
&0 &0 &0 &1 &0 &0 &0 &0 &0 &0 &0 &0 &0 &0 &0 &0 &0 &0 &0 &0 &0 &0 &0 &0 &0 &0 &0 &0 &0 &0 &0 &0 &0 &0 &0 \\
&0 &0 &0 &0 &1 &0 &0 &0 &0 &0 &0 &0 &0 &0 &0 &0 &0 &0 &0 &0 &0 &0 &0 &0 &0 &0 &0 &0 &0 &0 &0 &0 &0 &0 &0 \\
&0 &0 &0 &0 &0 &1 &0 &0 &0 &0 &0 &0 &0 &0 &1 &0 &0 &0 &0 &0 &0 &0 &0 &0 &1 &0 &0 &0 &0 &0 &0 &0 &0 &0 &1 \\
&0 &0 &0 &0 &0 &0 &1 &0 &0 &0 &0 &0 &0 &0 &0 &0 &0 &0 &0 &0 &0 &0 &0 &0 &1 &0 &0 &0 &0 &0 &0 &0 &0 &0 &1 \\
&0 &0 &0 &0 &0 &0 &0 &1 &0 &0 &0 &0 &0 &0 &0 &0 &0 &0 &0 &0 &0 &0 &0 &0 &0 &0 &0 &0 &0 &0 &1 &0 &0 &0 &1 \\
&0 &0 &0 &0 &0 &0 &0 &0 &1 &0 &0 &0 &0 &0 &0 &0 &0 &0 &0 &0 &0 &0 &0 &0 &0 &0 &0 &0 &0 &0 &0 &0 &0 &0 &0 \\
&0 &0 &0 &0 &0 &0 &0 &0 &0 &1 &0 &0 &0 &0 &1 &0 &0 &0 &0 &0 &0 &0 &0 &0 &0 &0 &0 &0 &0 &0 &1 &0 &0 &0 &1 \\
&0 &0 &0 &0 &0 &0 &0 &0 &0 &0 &1 &0 &0 &0 &0 &0 &0 &0 &0 &0 &0 &0 &0 &0 &1 &0 &0 &0 &0 &0 &1 &0 &0 &0 &0 \\
&0 &0 &0 &0 &0 &0 &0 &0 &0 &0 &0 &1 &0 &0 &0 &0 &0 &0 &0 &0 &0 &0 &0 &0 &0 &0 &0 &0 &0 &0 &1 &0 &0 &0 &1 \\
&0 &0 &0 &0 &0 &0 &0 &0 &0 &0 &0 &0 &1 &0 &1 &0 &0 &0 &0 &0 &0 &0 &0 &0 &1 &0 &0 &0 &0 &0 &1 &0 &0 &0 &0 \\
&0 &0 &0 &0 &0 &0 &0 &0 &0 &0 &0 &0 &0 &1 &0 &0 &0 &0 &0 &0 &0 &0 &0 &0 &1 &0 &0 &0 &0 &0 &0 &0 &0 &0 &1 \\
&0 &0 &0 &0 &0 &0 &0 &0 &0 &0 &0 &0 &0 &0 &0 &1 &0 &0 &0 &0 &0 &0 &0 &0 &1 &0 &0 &0 &0 &0 &1 &0 &0 &0 &1 \\
&0 &0 &0 &0 &0 &0 &0 &0 &0 &0 &0 &0 &0 &0 &0 &0 &1 &0 &0 &0 &0 &0 &0 &0 &0 &0 &0 &0 &0 &0 &1 &0 &0 &0 &1 \\
&0 &0 &0 &0 &0 &0 &0 &0 &0 &0 &0 &0 &0 &0 &0 &0 &0 &1 &0 &0 &0 &0 &0 &0 &0 &0 &0 &0 &0 &0 &0 &0 &0 &0 &0 \\
&0 &0 &0 &0 &0 &0 &0 &0 &0 &0 &0 &0 &0 &0 &0 &0 &0 &0 &1 &0 &0 &0 &0 &0 &0 &0 &0 &0 &0 &0 &0 &0 &0 &0 &0 \\
&0 &0 &0 &0 &0 &0 &0 &0 &0 &0 &0 &0 &0 &0 &0 &0 &0 &0 &0 &1 &0 &0 &0 &0 &1 &0 &0 &0 &0 &0 &0 &0 &0 &0 &0 \\
&0 &0 &0 &0 &0 &0 &0 &0 &0 &0 &0 &0 &0 &0 &0 &0 &0 &0 &0 &0 &1 &0 &0 &0 &0 &0 &0 &0 &0 &0 &1 &0 &0 &0 &1 \\
&0 &0 &0 &0 &0 &0 &0 &0 &0 &0 &0 &0 &0 &0 &0 &0 &0 &0 &0 &0 &0 &1 &0 &0 &0 &0 &0 &0 &0 &0 &0 &0 &0 &0 &0 \\
&0 &0 &0 &0 &0 &0 &0 &0 &0 &0 &0 &0 &0 &0 &0 &0 &0 &0 &0 &0 &0 &0 &1 &0 &1 &0 &0 &0 &0 &0 &1 &0 &0 &0 &1 \\
&0 &0 &0 &0 &0 &0 &0 &0 &0 &0 &0 &0 &0 &0 &0 &0 &0 &0 &0 &0 &0 &0 &0 &1 &0 &0 &0 &0 &0 &0 &1 &0 &0 &0 &1 \\
&0 &0 &0 &0 &0 &0 &0 &0 &0 &0 &0 &0 &0 &0 &0 &0 &0 &0 &0 &0 &0 &0 &0 &0 &0 &1 &0 &0 &0 &0 &1 &0 &0 &0 &0 \\
&0 &0 &0 &0 &0 &0 &0 &0 &0 &0 &0 &0 &0 &0 &0 &0 &0 &0 &0 &0 &0 &0 &0 &0 &0 &0 &1 &0 &0 &0 &0 &0 &0 &0 &0 \\
&0 &0 &0 &0 &0 &0 &0 &0 &0 &0 &0 &0 &0 &0 &0 &0 &0 &0 &0 &0 &0 &0 &0 &0 &0 &0 &0 &1 &0 &0 &0 &0 &0 &0 &0 \\
&0 &0 &0 &0 &0 &0 &0 &0 &0 &0 &0 &0 &0 &0 &0 &0 &0 &0 &0 &0 &0 &0 &0 &0 &0 &0 &0 &0 &1 &0 &1 &0 &0 &0 &0 \\
&0 &0 &0 &0 &0 &0 &0 &0 &0 &0 &0 &0 &0 &0 &0 &0 &0 &0 &0 &0 &0 &0 &0 &0 &0 &0 &0 &0 &0 &1 &0 &0 &0 &0 &0 \\
&0 &0 &0 &0 &0 &0 &0 &0 &0 &0 &0 &0 &0 &0 &0 &0 &0 &0 &0 &0 &0 &0 &0 &0 &0 &0 &0 &0 &0 &0 &0 &1 &0 &0 &1 \\
&0 &0 &0 &0 &0 &0 &0 &0 &0 &0 &0 &0 &0 &0 &0 &0 &0 &0 &0 &0 &0 &0 &0 &0 &0 &0 &0 &0 &0 &0 &0 &0 &1 &0 &0 \\
&0 &0 &0 &0 &0 &0 &0 &0 &0 &0 &0 &0 &0 &0 &0 &0 &0 &0 &0 &0 &0 &0 &0 &0 &0 &0 &0 &0 &0 &0 &0 &0 &0 &1 &1 \\
&0 &0 &0 &0 &0 &0 &0 &0 &0 &0 &0 &0 &0 &0 &0 &0 &0 &0 &0 &0 &0 &0 &0 &0 &0 &0 &0 &0 &0 &0 &0 &0 &0 &0 &0 \\
&0 &0 &0 &0 &0 &0 &0 &0 &0 &0 &0 &0 &0 &0 &0 &0 &0 &0 &0 &0 &0 &0 &0 &0 &0 &0 &0 &0 &0 &0 &0 &0 &0 &0 &0 \\
&0 &0 &0 &0 &0 &0 &0 &0 &0 &0 &0 &0 &0 &0 &0 &0 &0 &0 &0 &0 &0 &0 &0 &0 &0 &0 &0 &0 &0 &0 &0 &0 &0 &0 &0 \\
&0 &0 &0 &0 &0 &0 &0 &0 &0 &0 &0 &0 &0 &0 &0 &0 &0 &0 &0 &0 &0 &0 &0 &0 &0 &0 &0 &0 &0 &0 &0 &0 &0 &0 &0 \\
&0 &0 &0 &0 &0 &0 &0 &0 &0 &0 &0 &0 &0 &0 &0 &0 &0 &0 &0 &0 &0 &0 &0 &0 &0 &0 &0 &0 &0 &0 &0 &0 &0 &0 &0 \\
&0 &0 &0 &0 &0 &0 &0 &0 &0 &0 &0 &0 &0 &0 &0 &0 &0 &0 &0 &0 &0 &0 &0 &0 &0 &0 &0 &0 &0 &0 &0 &0 &0 &0 &0 \\
&0 &0 &0 &0 &0 &0 &0 &0 &0 &0 &0 &0 &0 &0 &0 &0 &0 &0 &0 &0 &0 &0 &0 &0 &0 &0 &0 &0 &0 &0 &0 &0 &0 &0 &0 \\
&0 &0 &0 &0 &0 &0 &0 &0 &0 &0 &0 &0 &0 &0 &0 &0 &0 &0 &0 &0 &0 &0 &0 &0 &0 &0 &0 &0 &0 &0 &0 &0 &0 &0 &0 \\
&0 &0 &0 &0 &0 &0 &0 &0 &0 &0 &0 &0 &0 &0 &0 &0 &0 &0 &0 &0 &0 &0 &0 &0 &0 &0 &0 &0 &0 &0 &0 &0 &0 &0 &0 \\
\end{matrix}$}
\end{center}
The last three non-zero rows are 
\[x+1,y,z+1\]
and by these we can reduce the polynomials arising from the rows of $M_{\leq 4}$ and getting a Gr\"obner basis. Hence, in this case the solving degree is $4$.
Observe that the polynomial $z+1$ could be obtained after the reduction of the polynomials in $M_{\leq 3 }$, multiplying the element $y$ by $z$, and summing to $yz+z+1$. This reduction is effectively a degree-$4$ relation.
\end{Example}

\subsection{Computing coefficients of the Hilbert Series and first fall degree}
As follows from the definition of the Hilbert Series, $\H_{R/(f_1,\ldots, f_m)}(d)$ corresponds to the total number of monomials of degree $d$ in $R$ minus the number of independent equations obtained by multiplying all degree $d-2$ monomials by the $f_i$. 
To find the number of independent equations, we should take in account that there are trivial relations among them. We first start from a characteristic-$0$ field $\KK$ and then we particularize the computation on the module $\FF_2[x_1,\ldots x_n]/(x_1^2+x_1,\ldots, x_n^2+x_n)$. \\
Let $\MM_{k,n}$ be the set of degree-$k$ monomials in $x_1,\ldots,x_n$. Moreover, given $\MF=\{f_1,\ldots,f_m\}$ we denote by $\MF_k$ the cardinality-$k$ subsets of $\MF$.\\
If $\cha(\KK)=0$, $d \in \NN$, $i \in \{1,\ldots,\lfloor \frac{d-2}{2}\rfloor \}$ and $\delta_i=d-2-2i$, then we have
\[
\Triv_{d,i} (f_1,\ldots, f_m;\KK) = \Bigg\{\mu \cdot \prod\limits_{f \in A} f \ | \  \mu \in \MM_{\delta_i,n}, \ \ A\in \MF_{i+1} \Bigg\},
\]
and the number of trivial relations at degree $d$ is 
\[
\triv_d(f_1,\ldots, f_m)= \sum\limits_{i=1}^{\lfloor\frac{d-2}{2}\rfloor} (-1)^{i-1} |\Triv_{d,i} (f_1,\ldots, f_m;\KK)|=\sum\limits_{i=1}^{\lfloor\frac{d-2}{2}\rfloor} (-1)^{i-1} \binom{n+\delta_i-1}{\delta_i} \cdot \binom{m}{i+1}.
\]
If $\KK=\FF_2$,  and we consider the field equations ${x_1^2+x_1,\ldots, x_n^2+x_n}$ we should add the conditions that $x_i^2=x_i$ and hence $f^2=f$ for any polynomial $f \in \FF_2[x_1,\ldots,x_n]$. These translate in considering the set $\ol{\MM}_{k,n}$ of the squarefree monomials in $x_1,\ldots,x_n$ and the set $\ol{\MF}_{k}$ of all the ``monomials'' in $f_1,\ldots, f_m$.
In this case for $d \in \NN $, $i \in \{1,\ldots,\lfloor \frac{d-2}{2}\rfloor \}$ and $\delta_i=d-2-2i$, we have
\[
\Triv_{d,i} (f_1,\ldots, f_m;\FF_2) = \Bigg\{\mu \cdot \prod\limits_{f \in A} f \ | \  \mu \in \ol{\MM}_{\delta_i,n}, \ \ A \in \ol{\MF}_{i+1} \Bigg\}
\]
and the number of trivial relations at degree $d$ in $\FF_2$ is 
\[
\triv_d(f_1,\ldots, f_m)= \sum\limits_{i=1}^{\frac{d-2}{2}} (-1)^{i-1} |\Triv_{d,i} (f_1,\ldots, f_m;\FF_2)|=\sum\limits_{i=1}^{\frac{d-2}{2}} (-1)^{i-1} \binom{n}{\delta_i} \cdot \binom{m+i}{i+1}.
\]

The degree of regularity defined in Section \ref{sec:HS&Alg} can be seen as the least degree $d$ such that the number of equations is greater than the number of monomials and  we can explicitly determine any degree-$d$ monomial. 
Another interesting invariant, that has been introduced in \cite{DG} and modified in \cite{DS2}, is the \emph{first fall degree}, $d_{\mathrm{fall}}$, namely the least degree for which there are relations between the degree-$d$ equations obtained from the system other than the ones described in the set $\Triv_d$ in this section. The first fall degree coincides with the maximum degree of polynomials in a solving algorithm such as $F_4$ and XL. Indeed, when a degree fall happens, one obtains a new polynomial of lower degree that reduces highest degree polynomials, producing new polynomials of lower degree, and so on.

\begin{Example}
    We consider the system $\MF$ of Example \ref{exa:MM}. The Macualay Matrix $M_{\leq3}$ produces the following set of polynomials
    \begin{align*}
   \{&x^3 + z^3 + z^2 + 1, \ \  x^2y + z^2 + 1, \ \  xy^2 + z + 1, \ \  y^3,\\
    &x^2z + z^3 + z + 1,  \ \ xyz + z^2 + z, \ \  y^2z + z + 1, \ \  xz^2 + z^3 + z^2 + z,\\
    &yz^2 + z^2 + 1, \ \  x^2 + z^2,  \ \ xy + z + 1, \ \  y^2,\\
    &xz + z^2 + z + 1,  \ \ yz + z + 1, \ \  x + z, \ \  y\}.
\end{align*}
hence, there are some degree falls and $d_{\mathrm{fall}}(\MF)=3$. Observe that reducing $yz + z + 1$ by $y$, one gets $z+1$, and reducing $x+z$ by $z+1$, one gets $x+1$, hence
finding a Gr\"obner basis for $\MF$.
\end{Example}

\section{Hilbert Series on Finite fields}\label{sec:HSFF}
In this section we study the case where $\KK=\FF_q$ and we take in account the field equations. We follow \cite[Section 3.2]{Ba} and generalize Proposition \ref{prop:regSeq2}.
 
We consider the finite field $\FF_q$ with $q=p^k$ for some prime $p$ and some $k \in \NN$. We consider $S=\FF_q[x_1,\ldots x_n]/(x_1^q-x_1,\ldots,x_n^q-x_n)$.
Although the above ring is not graded, because the ideal $(x_1^q-x_1,\ldots,x_n^q-x_n)$ is not homogeneous, one may observe that for $d\in \NN$ the degree-$d$ monomials of $S$ coincide with the degree-$d$ monomials of the graded ring $R=S^h=\FF_q[x_1,\ldots x_n]/(x_1^q,\ldots,x_n^q)$.
Hence
\[
\HS_{R}(t)=\frac{(1-t^{q})^n}{(1-t)^n}=(1+t+\ldots +t^{q-1})^n.
\]
Similarly to \cite[Definition 3.2.4]{Ba}, we give the definition of semi-regular sequence on $\FF_q$.

\begin{Definition}
     A sequence of homogeneous polynomials $f_1,\ldots, f_m$ is \emph{semi-regular} on $\FF_q$ if 
     \begin{itemize}
         \item $\MF=(f_1,\ldots, f_m) \neq R$;
         \item for any $i \in \{1,\ldots, m \}$ if $f_ig=0$ in $R/(f_1,\ldots, f_{i-1})$ and $\deg(f_ig)< i_{\reg{}}(\MF)$, then $g=0$ in $R/(f_1,\ldots,f_{i-1},f^{q-1}_i)$. 
     \end{itemize}
\end{Definition}

We now compute the Hilbert series of a semi-regular system of equations. 

\begin{Lemma}\label{lem:HFq}
    Let $\MF=(f_1,\ldots,f_m)$ be an ideal of $R$ generated by a semi-regular sequence. Let $\delta= \deg f_m$ and let $ \MF'=(f_1,\ldots, f_{m-1})$. Then for any $d \in \NN$ we have
    \[
    (R/\MF')_d =\bigoplus_{i=0}^{q-1} f^i (R/\MF)_{d-i\delta},
    \]
    and as a corollary
    \[
        \H_{R/\MF'}(d) =\sum_{i=0}^{q-1} \H_{R/\MF}(d-i\delta)
    \]
\end{Lemma}
\begin{proof}
We treat the case $m=1$, then the argument will inductively follow.
Let $f$ be a semi-regular element of $R$ and let 
\[
M=\bigoplus_{i=0}^{q-1} f^i (R/(f)).
\]
 We prove that $\phi: M \to R$ defined by 
 \[
\phi((g_0,\ldots , g_{q-1}))=\sum_{i=0} g_{i}f^i  \\
\]
is an isomorphism.

Any element of $h\in R$ can be written as $h=f^ig$ where $i \in \{0,\ldots, q-1\}$ is the maximum power of $f$ dividing $h$ and $f \not | g$, hence $g \in R/(f)$.
Moreover if $\phi((g_0,\ldots , g_{q-1}))=\phi((h_0,\ldots, h_{q-1}))$, namely
\[
 \sum_{i=0} g_{i}f^i  = \sum_{i=0} h_{i}f^i,
\]
for some $g_i,h_i \in R/(f)$
then
\[
 \sum_{i=0} (g_{i}-h_i)f^i =0
\]
We multiply by $f^{q-1}$ to obtain 
\[
f^{q-1} (g_0-h_0) =0
\]
From the semi-regularity of $f$ this implies $g_0-h_0 =fr$ for some $r \in R$ and then $g_0=h_0$ in $R/(f)$. By proceeding in this way one gets $g_i=h_i$ for any $i \in \{0,\ldots,q-1\}$.  Hence, if $\delta=\deg f$, and $d \in \NN$, we have 
\[
R_d=\bigoplus_{i=0}^{q-1} f^i (R/(f))_{d-\delta i}
\]

 \end{proof}
\begin{Theorem}
     Let $\MF=(f_1,\ldots,f_m)$ be an ideal of $R$ generated by a semi-regular sequence and let $d_i= \deg f_i$. Then
     \[
     \HS_{R/\MF}(t)=\Bigg[\frac{(1+t+\ldots +t^{q-1})^n}{\prod\limits_{i=1}^m (1+t^{d_i}+\ldots +t^{(q-1)d_i})}\Bigg]
     \]
\end{Theorem}
\begin{proof}
    In the notation of Lemma \ref{lem:HFq}, one has that 
    \[
    \sum\limits_d \H_{R/\MF'}(d)t^d = \sum\limits_d \Big(\sum_{i=0}^{q-1} \H_{R/\MF}(d-id_m)\Big)t^d=
    \]
    \[
    =\sum\limits_d\Big(\sum_{i=0}^{q-1} t^{id_m} \H_{R/\MF}(d)\Big)t^d= \sum\limits_d\Big(\sum_{i=0}^{q-1} t^{id_m}\Big) \H_{R/\MF}(d) t^d,
    \]
    that yields
    \[
     \sum\limits_d \H_{R/\MF}(d) t^d = \Bigg[\sum\limits_d \frac{1}{(1+t^{d_m}+\ldots +t^{(q-1)d_m})} \H_{R/\MF'}(d)t^d \Bigg]
    \]
    and hence
    \[
     \HS_{R/\MF}(t)= \Bigg[\frac{1}{(1+t^{d_m}+\ldots +t^{(q-1)d_m})}  \HS_{R/\MF'}(t) \Bigg].
    \]
    By inductively applying this argument one has 
    \[
     \HS_{R/\MF}(t)=\Bigg[\frac{1}{\prod\limits_{i=1}^m (1+t^{d_i}+\ldots +t^{(q-1)d_i})} \HS_{R}(t)\Bigg]= \Bigg[\frac{(1+t+\ldots +t^{q-1})^n}{\prod\limits_{i=1}^m (1+t^{d_i}+\ldots +t^{(q-1)d_i})}\Bigg]
    \]
    and the assertion follows. 
\end{proof}

\section{ Hilbert Series and Algebraic invariants of homogeneous OV-systems}\label{sec:HSOV}
In this section, inspired by \cite[Theorem 1]{IS} we give an exact expression for the Hilbert series of homogeneous OV-systems. In the section, we assume that $\KK$ has any characteristic, if not specified in the assertion. From now on, we consider polynomials in $R$ where $R$ is either $\KK[x_1,\ldots, x_n]$ when $\mathrm{char}(\KK)=0$ or $\FF_q[x_1,\ldots, x_n]/(x_1^q,\ldots, x_n^q)$ (in view of Section \ref{sec:HSFF}). To denote the latter case, we simply write $\KK=\FF_q$.  

\begin{Theorem}\label{thm:HFOV}
    Let $\MF$ be a homogeneous system of $m$ quadratic $OV$-equations with $v$ vinegar variables. Then for any $d \in \NN$
    \[
    \H_{R/\MF}(d)=\H_{\KK_o}(d) + \H_{\mathcal{V}/\MF}(d)  
    \]
\end{Theorem}
\begin{proof}
    Since $\MF \subseteq \mathcal{V}$, we consider the exact sequence 
   
    \[
    0\longrightarrow \mathcal{V}/\MF \longrightarrow R/\MF \longrightarrow R/\mathcal{V} \longrightarrow 0.
    \]
    From Lemma \ref{lem:addHS}, we have
    \[
     \H_{R/\MF}(d)=\H_{R/\mathcal{V}}(d)+ \H_{\mathcal{V}/\MF}(d).
    \]
      The assertion follows by observing that $R/\mathcal{V} \cong \KK[x_{v+1},\ldots,x_n]=\KK_o$.
\end{proof}
\begin{Corollary}
    Let $\MF$ be a homogeneous system of $m$ quadratic $OV$-equations with $v$ vinegar variables.
    \begin{itemize}
            \item If $\mathrm{char}(\KK)=0$, then for any $d \in \NN$
    \[
    \H_{R/\MF}(d)=\binom{n+d-1}{d} + \H_{\mathcal{V}/\MF}(d)  
    \]
      \item If $\KK=\FF_2$ and $\MF$ contains the field, then for any $d \in \NN$
    \[
    \H_{R/\MF}(d)=\binom{n}{d} + \H_{\mathcal{V}/\MF}(d)  
    \]   
    \end{itemize}
\end{Corollary}

In the following corollary we prove that the index of regularity of $R/\MF$ coincides with the index of regularity of $\MV/\MF$
\begin{Corollary}\label{cor:ireg}
     Let $\MF$ be a homogeneous system of $m$ quadratic $OV$-equations with $v$ vinegar variables. Then 
     \[
     i_{\reg{}}(R/\MF)=i_{\reg{}}(\MV/\MF)
     \]
\end{Corollary}

\begin{proof}
    Let $d=\dim R/\MF$ and $d'=\dim \MV/\MF$. From Theorem \ref{thm:HFOV} it follows that 
    \[
    \HS_{R/\MF}(t)=\frac{1}{(1-t)^{n-v}} + \HS_{\MV/\MF}(t)=\frac{1}{(1-t)^{n-v}} + \frac{h'(t)}{(1-t)^{d'}}.
    \]
    Since 
    \[
    d= \max\{\dim \MV/\MF,\dim \KK[x_{v+1},\ldots x_n] \}=\max\{d',n-v\},
    \]
    then we divide two cases:
    \begin{enumerate}
        \item $d=n-v$;
        \item $d=d'$.
    \end{enumerate}
    \noindent (1) If $d=n-v$, then we write
    \[
    \frac{h(t)}{(1-t)^{n-v}}=\HS_{R/\MF}(t)=\frac{1+(1-t)^{n-v-d'}h'(t)}{(1-t)^{n-v}},
    \]
    and $\deg h(t)= n-v-d' + \deg h'(t)$. Moreover, 
    \[
    i_{\reg{}}(R/\MF)=\deg h(t)-n-v+1 =\deg h'(t)-d'+1=i_{\reg{}}(\MV/\MF).
    \]
    \noindent If $d=d'$, then we write
    \[
     \frac{h(t)}{(1-t)^{d'}}=\HS_{R/\MF}(t)=\frac{(1-t)^{d'-(n-v)}+h'(t)}{(1-t)^{d'}}
    \]
    Hence, $\deg h(t)=\max\{d'-(n-v),\deg h'(t)\}$. If $\deg h'(t) < d'$, then both $i_{\reg{}}(R/\MF), i_{\reg{}}(\MV/\MF) \leq 0$. If  $\deg h'(t) > d' > d'-(n-v)$, then 
    \[
    i_{\reg{}}(R/\MF)=\deg h(t)-d'+1 =\deg h'(t)-d'+1=i_{\reg{}}(\MV/\MF).
    \]
    
\end{proof}
Motivated by Theorem \ref{thm:HFOV} and Corollary \ref{cor:ireg}, we give the following 

\begin{Corollary}
    Let $\MF$ be a homogeneous system of $m$ quadratic $OV$-equations with $v$ vinegar variables and assume that $\dim \mathcal{V}/\MF =0$. Then, for $d\geq d_{\reg{}}(\MF)$
    \[
    \H_{R/\MF}(d)=\H_{\KK_o}(d)
    \]
   
\end{Corollary}
\begin{proof}
    Since $\dim \mathcal{V}/\MF =0$, then $\HS_{\mathcal{V}/\MF}(t)=h'(t)$ with $\deg h'(t)=d_{\reg{}}(\MF)-1$, and hence  $\H_{\MV/\MF}(d)=0$ for $d \geq d_{\reg{}}(\MF)$.
\end{proof}

\subsection{Results on OV systems in characteristic 0}\label{sec:char0}
From now to the end of the section we assume that $\mathrm{char}(\KK)=0$.

We begin the section by computing the algebraic invariants of an OV-system.
\begin{Lemma}\label{lem:algInv}
    Let $f_1,\ldots, f_m$ be a homogeneous OV-polynomials with $v$ vinegar variables and let $\MF=(f_1,\ldots, f_m)$. Then
    \begin{enumerate}
    \item $\height(\MF)\leq v$;
    \item if $f_1,\ldots f_m$ is regular then $m\leq v$;
    \item $\dim R/\MF = \max\{n-v,n-m\}$;
    \end{enumerate}
\end{Lemma}
\begin{proof}
    (1) Since $\MF\subseteq \MV=(x_1,\ldots, x_v)$ and $\height(\MV)=v$, then the assertion follows.\\
    (2) From the fact that $\MF \subseteq \MV $, we obtain $\dim R/\MF\geq\dim R/\MV = n-v$.
    If $m>v$, then from Proposition \ref{prop:regSeq} $\dim R/\MF=n-m < n-v$ that contradicts $\dim R/\MF\geq n-v$. \\
    (3) Follows from Proposition \ref{prop:regSeq} and (2).\\
\end{proof}
From Lemma \ref{lem:algInv}.(2) we have that for $m\geq v+1$ the OV-system $\MF=\{f_1,\ldots,f_m\}$ is not regular, hence the colon ideals $(f_{1},\ldots,f_{i}): f_{i+1}$ are non trivial for $i\geq v$. In the following Proposition, we give some informations on these ideals.

\begin{Proposition}\label{prop:v+2}
Let $\MF=(f_1,\ldots, f_v)$ be a homogeneous $OV$-system of $v$ equations in $n$ variables with $f_1,\ldots, f_v$ regular sequence. Then for any $g \in \MV$ with $\deg g =2$ there exists $h \in R$ with $\deg h=v$ such that $hg \in \MF$.
\end{Proposition}
\begin{proof}
Any degree-$2$ polynomial $g \in \MV$ can be written as
\[
g= \sum\limits_{j=1}^{v} x_j g_{j}
\] 
where $g_j$ is a degree-1 polynomial.
Since $\MF \subset \MV$, then for any $i \in \{1,\ldots v\}$ one has
\[
f_i= \sum\limits_{j=1}^{v} x_{j}f_{ij},
\]
and we call $A$ the $v\times v$ matrix $(f_{ij})$. Observe that $\det(A)$ is a degree-$v$ polynomial.
We claim that for any $g \in \MV$ with $\deg g =2$ one has
\[
\det(A) g \in \MF. 
\]
In particular, if $A_{i}$ is the matrix $A$ in which we replace the $i$-th column by $f_1,\ldots f_v$, from linear algebra relations we obtain that for any $i \in \{1,\ldots v\}$ we have
\[
x_i\det(A)=\det(A_i)\in \MF, 
\]
hence 
\[
\det(A) g= \sum\limits_{j=1}^{v} \det(A)x_j g_{j}= \sum\limits_{j=1}^{v} \det(A_j)g_j \in \MF 
\] 
The claim follows if we prove that $0 \neq \det(A) \notin \MF$.
If one has $\det(A)=0$, then there exists a linear relation between the rows of the matrix, say for any $j \in \{1,\ldots, v\}$, one has
\[
\sum_{j=1}^{v} \lambda_{j}f_{ji}=0
\]
with $\lambda_i \in \KK$ without loss of generality $\lambda_v \neq 0$ 
the latter turns into 
\[
\sum_{j=1}^{v} \lambda_{j}f_{j}=0
\]
that implies $\lambda_{v}f_v \in (f_1,\ldots, f_{v-1})$, yielding that the sequence is not regular.
We prove by induction on  $v$ that $\det(A) \notin \MF$. In  the following we assume that $x_1\notin \supp(f_{ij})$ if $j>1$.
If $v=2$, then $f=x_1f_1+x_2f_2,$ and $g=x_1g_1+x_2g_2$. We have $\det(A)=f_1g_2-f_2g_1$. If $\det(A) \in (f,g)$ and since $\deg \det(A)=2$, then there exist $\lambda,\mu \in \KK$ such that $\det(A)=\lambda f+ \mu g$. We have
\[
\lambda x_1 f+ \mu x_1 g =x_1\det(A)=\det(A_1)=fg_2-f_2g,
\]
yielding 
\[
(x_1\lambda -g_2)f=-(x_1\mu+f_2)g,
\]
with $(x_1\lambda -g_2)\notin (g)$, because $\deg g_2 =1$ and $\deg(g)=2$, and $(x_1\lambda -g_2),(x_1\mu +f_2)\neq 0$ by our assumption on $f_2,g_2.$ Hence, $f,g$ is not a regular sequence, and this is a contradiction. We now assume, by induction, the assertion true for any sequence of $v-1$ polynomials, and we prove that for $v$.

If one has $\det (A) \in\MF$, say 
\[
\det(A)= \sum_{k=1}^v \lambda_k f_k
\]
for some $\lambda_k \in R$ with $\deg \lambda_k=v-2$,  one has 
\[
\det(A_1)=\sum_{k=1}^v \mu_{k} f_k
\]
for some $\mu_k \in R$ with $\deg \mu_k=v-1$, then
\[
x_1 \det(A)= \sum_{k=1}^v x_1\lambda_k f_k.
\]
Since $x_1\det(A)=\det(A_1)$, then
\[
(\mu_{v}-x_1\lambda_v)f_v =\sum_{k=1}^{v-1} (\mu_{ik}-x_i\lambda_k)f_k \in (f_1,\ldots, f_{v-1}).
\]
We recall that $\mu_v= \det(A_{v1}),$ that is the complement of the element $f_{v1}$ in $A$, that has coefficients in $f_{ij}$ with $1\leq i\leq v-1$, $2\leq j\leq v$. Our assumption that $x_1\notin \supp(f_{ij})$ if $j>1$, implies  $(\mu_{v}-x_1\lambda_v)\neq 0$. If $\mu_{v}-x_1\lambda_v \in (f_1,\ldots,f_{v-1})$,
then
\[
\mu_{v}-x_1\lambda_v = \sum_{i=1}^{v-1}\alpha_{i}f_i=\sum_{i=1}^{v-1}\sum_{j=1}^v \alpha_i x_j f_{ij}
\]
and
\[
\mu_{v}- \sum_{i=1}^{v-1}\sum_{j=2}^v \alpha_i x_j f_{ij}= \rho x_1
\]
If $\rho\neq 0$, then a non trivial linear combination on $f_{ij}$ is in $(x_1)$ and this is a contradiction to the assumption $x_1 \notin \supp(f_{ij})$.
If $\rho= 0$, then 
\[
\mu_v\in (f_{1}-x_{1}f_{11},\ldots,f_{v-1}-x_{1}f_{1v-1}),
\]
contradicting the inductive hypothesis.
\end{proof}

\begin{Remark}\label{rmk:DeltaS}
    From Proposition \ref{prop:v+2}, given an OV-system of $m$ equations any cardinality-$v$ subset $S$ that is a regular sequence, we find a degree-$v$ element $\Delta_S$ such that $\Delta_S g$ is in the ideal generated by $S$.
\end{Remark}

\begin{Corollary}\label{cor:TRv+2}
    Let $\MF$ be a homogeneous OV-system of $v < m < n$ equations in $n$ variables with $v$ vinegar variables, such that any cardinality-$v$ subset is regular, and let $G$ be a cryptographic semiregular system of $m$ equations in $n$ variables.
    Then
    \[
    \H_{R/\MF}(v+2)= \H_{R/G}(v+2)+\binom{m}{v+1}.
    \]
\end{Corollary}
\begin{proof}
    According to Proposition \ref{prop:v+2}, any set of $v+1$ OV quadratic polynomials gives rise to a relation of degree $v+2$. Hence, the number of such relations is $\binom{m}{v+1}$.
\end{proof}

We give a first definition of semi-regular sequence of $OV$-polynomials.
\begin{Definition}\label{def:HFG}
   Let $G$ be a cryptographic semi-regular homogeneous quadratic system of $m$ equations in $n$ variables and let $\MF$ be a homogeneous OV-system of $m$ equations in $n$ variables. We say that  $\MF$ is \emph{semi-regular} if for $0\leq d\leq v+1$ we have 
    \[
    \H_{\MV/\MF}(d)=
        \max\Bigg\{\H_{R/G}(d)-\binom{n-v+d-1}{d},0\Bigg\}
    \]
\end{Definition} 
We restrict our attention to the interval $0\leq d\leq v+1$, because in view of Proposition \ref{prop:v+2}, for $d\geq v+2$ one has that $\dim_\KK G_d \geq \dim_\KK \MF_d$ and  then $\H_{\MV/\MF}(d)>\H_{R/G}(d)-\binom{n-v+d-1}{d}$.

In the case of overdetermined system, we can bound the degree of regularity.
\begin{Proposition}\label{prop:dreg}
    Let $f_1,\ldots, f_m$ be a semi-regular system of homogeneous OV-polynomials with $v$ vinegar variables and let $\MF=(f_1,\ldots, f_m)$ with $m\geq n$. Then $d_{\reg{}}(\MF)\leq v+1$.
\end{Proposition}
\begin{proof}
    To prove the claim, we prove that $d_{\reg{}}=v+1$ in the case $m=n$, then from Remark \ref{rem:FF'} it is clear that the more the number of equations increases, the more  the degree of regularity decreases.
    Let $m=n$. A cryptographic semiregular system $G$ of $n$ equations in $n$ variables has Hilbert series 
    \[
    \frac{(1-t^2)^n}{(1-t)^n}=(1+t)^n,
    \]
    that is for any $0\leq d\leq n$ $\H_{R/G}(d)=\binom{n}{d}$. On the other hand, from Definition \ref{def:HFG}
    \[
    \H_{\MV/\MF}(v+1)=\H_{R/G}(v+1)-\binom{n-v+v+1-1}{v+1}=\binom{n}{v+1}-\binom{n}{v+1}=0,
    \]
    that is $d_{\reg{}}\leq v+1$. Furthermore, 
    \[
    \H_{\MV/\MF}(v)=\H_{R/G}(v)-\binom{n-1}{v}=\binom{n}{v}-\binom{n-1}{v}=\binom{n-1}{v-1}>0.
    \]
    Hence, $d_{\reg{}}=v+1$.
\end{proof}

\begin{Remark}\label{rmk:Sd}
    It is worth noting that in Definition \ref{def:HFG} we consider the degree of regularity as the minimum $d$ for which $\dim_\KK \MF_d \geq \dim_\KK \MV_d$. If, for some $d$, we have that $\dim_\KK \MF_d$ and $\binom{n-v+d-1}{d}$ are very close, then it may happen that either $d_{\reg{}}=d$ or $d_{\reg{}}=d+1$. This depends on the saturation of the space $\SS(d)\subseteq \MV_d$ of the monomials effectively appearing in the rows of the matrix $M_d$, that is slightly smaller than $\MV_d$.
\end{Remark}
\begin{Example}
    We consider homogeneous OV-systems of $m=12$ equations in $n=9$ variables with $v=3$ vinegar variables. A cryptographic semiregular system $G$ with the same parameters is such that 
    \[
    \H_{R/G}(3)=57, 
    \]
    \[
    \HS_{\KK[x_1,\ldots,x_6]}(t)=1 + 6t + 21t^2 + 56t^3 + 126t^4 + 252t^5 + 462t^6 +
    \]
    we see that at degree $3$ the coefficient $56$ is very close to $57$. 
    In this case, the degree of regularity of an OV-system is either $3$ or $4$. In fact we consider two different systems $\MF_1$ and $\MF_2$, their Hilbert series are 
    \[
     \HS_{R/\MF_1}(t)=1 + 9t + 33t^2 + 57t^3 + 126t^4 + 252t^5 + 462t^6 +\ldots
    \]
    \[
    \HS_{R/\MF_2}(t)=1 + 9t + 33t^2 + 57t^3 + 127t^4 + 253t^5 + 463t^6 + \ldots 
    \]
    In the above cases, we have
     \[
     \HS_{\MV/\MF_1}(t)= 3t + 12t^2+t^3
    \]
    \[
    \HS_{\MV/\MF_2}(t)=3t + 12t^2 + t^3 + t^4 + t^5 + t^6 + \ldots 
    \]
    In the first case the degree of regularity is $4$ and in the second case is $3$. This difference of behavior is due to $\dim_\KK \MV_3=\binom{9+3-1}{3}-56=109$ and $\dim_\KK G_3=108$. If there is a monomial missing in $\SS(3)$, then $G_3=\SS(3)$, else $G_3 \neq \SS(3)$.
 \end{Example}

Inspired by the results of this section, in particular by Proposition \ref{prop:v+2} and Proposition \ref{prop:dreg},  we give the following definition for the semi-regularity of the sequence in terms of their Hilbert Series.
\begin{Definition}\label{def:OVsemi}
Let $\MF$ be a homogeneous $OV$-system of $m$ equations in $n$ variables. Then we say that $\MF$ is semi-regular if 
\begin{itemize}
    \item if $m\leq v$ then it is regular and 
    \[
    \HS_{R/\MF}(t)=\frac{(1-t^2)^m}{(1-t)^n}
    \]
    \item for $v+1\leq m < n$ 
     \[
     \HS_{R/\MF}(t)=\frac{(1-t^2)^m}{(1-t)^n} +  s(t) t^{v+2},
    \]
    where $s(t)$ is the series of a module arising from the degree $v$ relations $\{\Delta_{S} : S \subset \{f_1,\ldots,f_m\}, \ |S|=v\}$, where $\Delta_S$ is defined as in Remark \ref{rmk:DeltaS}.
    \item if $m \geq n$
    \[
    \HS_{R/\MF}(t)=\max \Bigg\{\Bigg[\frac{(1-t^2)^m}{(1-t)^n} \Bigg],\frac{1}{(1-t)^{n-v}}\Bigg\}
    \]
    where $\max\{\sum_i a_i,\sum_i b_i\}=\sum_i c_i$ where $c_i=\max\{a_i,b_i\}$.
\end{itemize}
\end{Definition}

\subsection{Hilbert series of quadratic systems}\label{sec:HSQ}
In this section we use Noether Normalization and Relative Finiteness to study the Hilbert series of any quadratic system. We show that any quadratic system $\MF$ can be seen as an OV-system with respect to $n-d$ variables where $d=\dim R/\MF$. We assume that $\mathrm{char}(\KK)=0$.

\begin{Proposition}
    Let $\MF$ be a homogeneous ideal with $\dim R/\MF=d$. Then there exists a map 
    \[
    \phi: \KK[x_1,\ldots, x_n] \to \KK[y_1,\ldots, y_n]
    \]
    such that $\phi (\MF) \subset \MV=(y_1,\ldots, y_{n-d})$ and for any $k \in \NN$
    \[
    \H_{R/\MF}(k)= \H_{\MV/\phi(\MF)}(k)+\binom{d+k-1}{k}
    \]
\end{Proposition}
\begin{proof}
    The existence of $\phi$ is guaranteed by Noether Normalization, since the field $\KK$ has characteristic 0 and hence it is infinite. In particular, 
    \[
    \KK[y_{n-d+1},\ldots, y_{n}] \to \KK[y_1,\ldots, y_n]/\phi(\MF)
    \]
    is injective and finite, that is $\MF\subset \MV$. In this way, $\phi(\MF)$ is a system of OV-equations with $v=n-d$ vinegar variables. Hence the second part follows from Theorem \ref{thm:HFOV}.
\end{proof}

\begin{Example}
    Let $\MF$ be the following $4$ quadratic polynomials on $\FF_2[x_1,x_2,x_3,x_4]$:
    \begin{align*}
    &x_2x_3 + x_2x_4 + x_3x_4 + x_4^2\\
    &x_1^2 + x_1x_2 + x_1x_4 + x_2x_4\\
    &x_2x_3 + x_1x_4\\
    &x_2x_3 + x_2x_4
    \end{align*}
    The dimension of the ideal generated by $\MF$ is $2$. We consider the following change of coordinates 
    \[
    \begin{cases}
        x_1=y_2+y_3\\
        x_2=y_3\\
        x_3=y_1+y_4\\
        x_4=y_4
    \end{cases}
    \]
    The polynomials become 
    \begin{align*}
    &y_1y_3+y_1y_4\\
    &y_2^2+y_2y_3+y_2y_4\\
    &y_1y_3+ + y_2y_4\\
    &y_1y_3
    \end{align*}
    that are $OV$ polynomials with vinegar variables $y_1$ and $y_2$. The Hilbert series of $\MF$ is 
    \[
    1 + 4t + 6t^2 + 7t^3 +\ldots +(k+4)t^k + \ldots,
    \]
    and this always greater than or equal to $\binom{2+k-1}{k}=k+1$ for any $k \geq 0$.
\end{Example}

\section{Hilbert series of Mixed Systems}\label{sec:Mix}
In this section, we focus our attention on a MQ-system that consists of homogeneous OV-polynomials as well as homogeneous fully quadratic polynomials. 
We set $u$ to be the number of fully quadratic polynomials $q_1,\ldots, q_u$ and we set $Q=(q_1,\ldots,q_u)$. We consider an OV-system $\MF=(f_1,\ldots,f_e)$ and we set $\PP=(f_1,\ldots,f_{e},q_1\ldots q_u)$ and $m=e+u$. We want to compute the Hilbert series of $\PP$. Actually, we have several ways to compute such a series. 
For this aim, we consider the ideal $\MV+\QQ$, namely the ideal generated by the vinegar variables and $q_1,\ldots, q_u$. We have
\[
\MV+\QQ=\MV+\QQ_{o}
\]
where $\QQ_o$ is the ideal generated by the ``oil'' part of $q_1,\ldots, q_u$. 

\begin{Theorem}\label{thm:HFP}
    Let $\PP = \MF + Q$ be a mixed system. Then for any $d \in \NN$,
    \begin{enumerate}
        \item $\H_{R/\PP}(d)=\H_{(\MV+\QQ)/\PP}(d)+ \H_{\KK_o/\QQ_o}(d)$;
        \item $\H_{R/\PP}(d)=\H_{R/\MF}(d)-\H_{\QQ/\MF\cap \QQ}(d)$;
        \item $\H_{R/\PP}(d)=\H_{R/\QQ}(d)-\H_{\MF/\MF\cap \QQ}(d)$.
    \end{enumerate}
\end{Theorem}
\begin{proof}
(1) Similarly to Theorem \ref{thm:HFOV}, we consider the following exact sequence 
\[
0 \longrightarrow (\MV+Q)/\PP \longrightarrow R/\PP \longrightarrow  \KK_o/\QQ_o \longrightarrow 0,
\]
and the assertion follows from Lemma \ref{lem:addHS}.\\
(2) Since $\MF \subset \PP$, then the following exact sequence arises:
\[
0 \longrightarrow \PP/\MF \longrightarrow R/\MF \longrightarrow  R/\PP \longrightarrow 0,
\]
From the isomorphism theorem, one has that 
\[
\PP/\MF=(\MF+\QQ)/\MF \cong \QQ/(\MF\cap \QQ),
\]
and again the thesis follows by applying Lemma \ref{lem:addHS}.\\
(3) Follows similarly as (2).
\end{proof}
We call \emph{degree of regularity} $d_{\reg{}}(\PP)$ the minimum integer for which $\PP_d=(\MV+\QQ_o)_d$, namely $i_{\reg{}}((\MV+\QQ_o)/\PP)$.
\begin{Example}
We consider a system $\PP$ of $12$ equations in $10$ variables, where $\MF$ is an OV system of $6$ equations with $3$ vinegar variables and hence $u=6$.
        The series of the system $\PP$ is 
        \[
       1 + 10t + 43t^2 + 100t^3 + 121t^4 + 63t^5 + 64t^6 + 64t^7 +\ldots 
        \]
        A cryptographic semiregular system of $10$ equations has Hilbert series
        \[
       1 + 10t + 43t^2 + 100t^3 + 121t^4 + 22t^5+\ldots 
        \]
        while the Hilbert series of the system $\QQ_o$ of $6$ equations in $10-3=7$ variables is
        \[
         1 + 7t + 22t^2 + 42t^3 + 57t^4 + 63t^5 + 64t^6 + 64t^7 + \ldots
        \]
        Since $22 < 63$, then the degree of regularity is $5$.

\end{Example}
In general, the degree of regularity corresponds to the minimum $d$ for which $\H_{R/G}(d)-\H_{\KK_o/\QQ_o}(d)\leq 0$, and hence $\H_{R/G}(d)\leq \H_{\KK_o/\QQ_o}(d)$, where $G$ is a cryptographic semiregular system of $m$ equations in $n$ variables. Indeed, one has
\[
\H_{(\MV+Q)/\PP}(d)= \H_{R/G}(d)-\H_{\KK_o/\QQ_o}(d),
\]
and
\[    
\dim_{\KK}(\MV+\QQ)_d=\dim_\KK R_d -\H_{\KK_o/\QQ_o}(d).
\]
This may not be the case if the system $\MF$ is underdetermined as showed in the next example.

\begin{Example}
        We consider a system $\PP$ of $10$ equations in $10$ variables, where $\MF$ is an OV system of $6$ equations with $3$ vinegar variables and hence $u=4$.
        The series of the system $\PP$ is 
        \[
        1 + 10t + 45t^2 + 120t^3 + 210t^4 + 267t^5 + 297t^6 + 354t^7 + 457t^8 + 585t^9+ \ldots 
        \]
        A cryptographic semiregular system of $10$ equations has Hilbert series
        \[
        1 + 10t + 45t^2 + 120t^3 + 210t^4 + 252t^5 + 210t^6 + 120t^7 + 45t^8+10t^9+t^10
        \]
        while the Hilbert series of the system $\QQ_o$ of $4$ equations in $7$ variables is
        \[
         1 + 7t + 24t^2 + 56t^3 + 104t^4 + 168t^5 + 248t^6 + 344t^7 + 456t^8+585t^9+ \ldots
        \]
        
        The degree of regularity of such system is $8$. In fact, by subtracting the series of $\KK_{o}/\QQ_o$ from the series $R/\PP$, one obtains
        \[
        3t + 21t^2 + 64t^3 + 106t^4 + 99t^5 + 51t^6 + 10t^7 + t^8 + t^9 \ldots 
        \]
        Notice that the system starts to differ from the cryptographic semiregular system at degree $5$ and $\dim _\KK G_6 > \dim_\KK (\MV+\QQ)_6$, but $\dim_\KK \PP_6 < \dim_\KK (\MV+\QQ)_6$, and this is due to the trivial relations that arise from the degree $5=v+2$. Indeed, from Proposition \ref{prop:v+2} and  Corollary \ref{cor:TRv+2}, we have 
        \[
        \binom{e}{v+1}=\binom{6}{4}=15
        \]
        trivial relations at the degree $5$.
    
\end{Example}

\section{Solving degree and first fall degree of non-homogeneous OV systems}\label{sec:HSNH}
In this section, we analyze the case of non-homogeneous OV (quadratic) systems. Namely, systems $\MF$ of non-homogeneous polynomials in which $\MF^\top$ is an OV system. In this case, we say that $\MF$ is an \emph{NH-OV system}.
 \begin{Definition}
    Given a system a NH-OV system $\MF$, we define \emph{degree of regularity} $d_{\reg}(\MF) = i(\MF^\top)$.
\end{Definition}

From the above Definition and from Proposition \ref{prop:dreg}, the degree of regularity is bounded by $v+1$ in the determined case. We analyze the behaviour of both the solving degree and first fall degree.
We premise the following lemmas on the Macaulay Matrix.
\begin{Lemma}\label{lem:SD}
    Let $G$ be a cryptographic semiregular system of $m$ equations in $n$ variables, and let $H(t)=\sum\limits_{i} h_i t^i$, with $h_i \in \ZZ$, be its Hilbert series.  Let $d$ be the minimum positive integer such that 
    \[
    \sum_{i=0}^d h_i\leq 0,
    \]
    then $\mathrm{solv.deg}(G)\leq d$.
\end{Lemma}
\begin{proof}
    The solving degree is the minimum $k$ such that $\dim_\KK \ker M_{\leq k} =0$. We observe that for any $k$, 
    \[
    \dim_\KK \ker M_{\leq k}=\sum_{i=0}^k \dim_\KK R_i - \dim_\KK M_i = \sum_{i=0}^k h_k
    \]
\end{proof}

By the previous Lemma we obtain the following
\begin{Proposition}\label{prop:sdOV}
        Let $\MF$ be a NH-OV system of $m$ equations in $n$ variables with $v$ vinegar variables, and let $H(t)=\sum\limits_{i} h_i t^i$ with $h_i \in \ZZ$ be the Hilbert series of a cryptographic semiregular system $G$ of $m$ equations in $n$ variables. Let $d$ be the minimum positive integer such that 
    \[
    \sum_{i=0}^d h_i - \H_{\KK_o}(d)\leq 0,
    \]
    then $\mathrm{solv.deg}(\MF)\leq d$.
\end{Proposition}
\begin{proof}
      The solving degree corresponds is less than or equal to the minimum $d$ such that $\dim_\KK\ker M_{\leq d}(\MF)=0$. Since the degree-$d$ monomials in $\KK_{o}$ do not appear in the system, there are $\H_{\KK_o}(d)$ zero columns in $M_d(\MF)$ that can be removed.  Hence, if  $G$ is a cryptographic semiregular system of $m$ equations in $n$ variables, there are $\H_{\KK_o}(d)-\rk M_d(G)$ relations in $M_{\leq d}$ that are degree fall and can be used to obtain relations on monomials of lower degree. In this case one has that the solving degree is less than or equal to the minimum $d$ such that 
   \[
   \H_{\KK_o}(d)-\rk M_d \geq \sum_{k=0}^{d-1} \H_{R/G} (k) \Rightarrow \H_{\KK_o}(d)\geq \sum_{k=0}^{d} \H_{R/G} (k),
   \]
   and the assertion follows. 
\end{proof}
\begin{Proposition}
    Let $\mathrm{char}(\KK)= 0$ and let $\MF$ be a NH-OV system of $m$ equations in $n$ variables with $m\geq n+1$ such that $\MF^top$ is semi-regular. Then
    \[
    \mathrm{solv.deg} (\MF) \leq v+1.
    \]
\end{Proposition}
\begin{proof}
   We prove that in the case $m=n+1$, we have $\mathrm{solv.deg} (\MF) \leq v+1$. 
 From Proposition \ref{prop:sdOV}, we have to compute the least $d$ for which 
  \[
   \binom{n-v+d-1}{d}\geq \sum_{k=0}^{d} \H_{R/G} (k),
   \]
   where $G$ is a cryptographic semiregular system.
   We recall that 
   \[
   \HS_{R/G}(t)=\frac{(1-t^2)^{n+1}}{(1-t)^n}=(1+t)^{n+1} (1-t)
   \]
   Hence $\H_{R/G}(d)=\binom{n+1}{d}-\binom{n+1}{d-1}$ for $d\geq 1$.
   Observe that 
   \[
   \sum_{k=0}^{d} \H_{R/G} (k)= \binom{n+1}{d}
   \]
   and hence
   \[
 \binom{n-v+d-1}{d}\geq\binom{n+1}{d}
   \]
  if $d=v+1$.
\end{proof}

Even if in general there are no relations between first fall degree, degree of regularity, and solving degree as shown in many examples (see \cite{CG,DS2}), by restricting our focus to quadratic systems, we establish the following bounds.
\begin{Proposition}
    Let $G$ be a non-homogeneous quadratic system of $m$ equations in $n$ variables, with $m\geq n$.  Then
    \begin{itemize}
        \item $d_{\mathrm{fall}}(G) \leq d_{\reg{}}(G)+1$;
        \item $d_{\reg{}}(G) \leq \mathrm{solv.deg}(G)$.
    \end{itemize}
    \end{Proposition}
\begin{proof}
    (1) Follows from the fact that, given $\HS_{R/G}(t)=\sum\limits_{i} h_i t^i$, $d_{\reg{}}(G)$ is the minimum $d$ such that $h_d \leq 0$. In the case $h_d=0$ there is no degree fall at degree $d$ and $d_{fall}(G)=d_{*}(G)+1$.\\
    (2) Follows from Proposition \ref{lem:SD}. and from the fact that, given $\HS_{R/G}(t)=\sum\limits_{i} h_i t^i$, $d_{\reg{}}(G)$ is the minimum $d$ such that $h_d \leq 0$.
\end{proof}

 \subsection*{Acknowledgments}
 This paper was made possible through the financial assistance provided by the Project ECS 0000024 “Ecosistema dell’innovazione - Rome Technopole” financed by EU in NextGenerationEU plan through MUR Decree n. 1051 23.06.2022 PNRR Missione 4 Componente 2 Investimento 1.5 - CUP H33C22000420001.

\end{document}